\documentclass[12pt]{amsart}
\usepackage[utf8]{inputenc}

\usepackage{fullpage, graphicx, subcaption}
\usepackage{hyperref}
\usepackage[ruled,vlined]{algorithm2e}

\graphicspath{ {figs/} }

\newtheorem{thm}{Theorem}[section]
\newtheorem{prop}[thm]{Proposition}

\newtheorem{lemma}[thm]{Lemma}

\theoremstyle{definition}
\newtheorem{definition}[thm]{Definition}

\title{Subdivergence-free gluings of trees}
\author{Xinle Dai}
\address{Department of Mathematics, Harvard University, Cambridge, Massachusetts, USA}
\email{xdai@math.harvard.edu}
\author{Jordan Long}
\author{Karen Yeats}
\address{Department of Combinatorics and Optimization, University of Waterloo, Waterloo, Canada}
\email{jmzlong@uwaterloo.ca}
\email{kayeats@uwaterloo.ca}
\thanks{JL was supported by an NSERC USRA during this work.  KY is supported by an NSERC Discovery grant and the Canada Research Chairs program.  KY thanks Dirk Kreimer for discussions of cut graphs in quantum field theory.  This work was done partly while XD was at University of Waterloo}

\begin{document}

\begin{abstract}
 A gluing of two rooted trees, also known as a tanglegram, is an identification of their leaves and un-subdivision of the resulting 2-valent vertices.  A gluing of two rooted trees is subdivergence-free if it has no 2-edge cuts with both roots on the same side of the cut.  The problem and language is motivated by quantum field theory.  We enumerate subdivergence-free gluings for certain families of trees, showing a connection with connected permutations, and we give algorithms to compute subdivergence-free gluings.
\end{abstract}

\maketitle

\section{Introduction}

\subsection{Set up}

We are interested in a purely combinatorial question about rooted trees that arose from a question in quantum field theory.

Let $t_1$ and $t_2$ be two rooted trees with the same number of leaves.  Let $n$ be the number of leaves.  There are $n!$ ways to identify the leaves of $t_1$ with the leaves of $t_2$.  We view each such identification as a graph in almost the obvious way, but with one important modification: the identified leaves are identified as graph vertices, resulting in a 2-valent vertex for each identified pair of leaves, as in \ref{fig:gluing_trees_unfinished}, and then one of the incident edges of each of these 2-valent vertices is contracted, resulting in each identified pair of leaves along with their two incident edges becoming one edge in the graph, as in \ref{fig:gluing_trees_finished}. The graph has two distinguished vertices corresponding to the roots of $t_1$ and $t_2$.  These $n!$ graphs are the \emph{gluings} of $t_1$ and $t_2$.

Another way to set up the definition of the gluings of $t_1$ and $t_2$ is to define the rooted trees so that their leaves are unpaired half edges.  Then the identification of leaves of $t_1$ with leaves of $t_2$ simply pairs the corresponding half edges into edges, directly building the graph described above.

Such gluings of trees have been studied in the context of phylogenetics and then in combinatorics under the name of \emph{tanglegrams} \cite{10.1093/bioinformatics/btr210, 5551100, BILLEY2017239, ralaivaosaona_et_al:LIPIcs.AofA.2018.32, GESSEL2021105498}.

We say a gluing of $t_1$ and $t_2$ has a \emph{subdivergence} if it has a 2-edge cut with the property that $t_1$ and $t_2$ each contain one edge of the cut as a non-leaf edge and one of the components of the cut doesn't contain a root vertex. In the tanglegram literature, this notion is called \emph{irreducibility} \cite{ralaivaosaona_et_al:LIPIcs.AofA.2018.32, BLACK2023102550}.  For example, Figure \ref{fig:subdivergence} is an example of a gluing with a subdivergence, while Figure \ref{fig:gluing_trees_finished} is a subdivergence-free gluing.

\begin{figure}
\begin{subfigure}{.3\textwidth}
  \centering
  \includegraphics[]{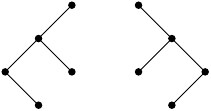}
  \caption{}
  \label{fig:rooted_trees}
\end{subfigure}%
\begin{subfigure}{.3\textwidth}
  \centering
  \includegraphics[]{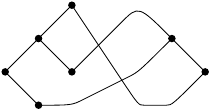}
  \caption{}
  \label{fig:gluing_trees_unfinished}
\end{subfigure}%
\begin{subfigure}{.3\textwidth}
  \centering
  \includegraphics[]{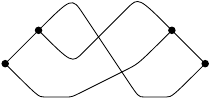}
  \caption{}
  \label{fig:gluing_trees_finished}
\end{subfigure}
    \caption{A gluing of two rooted trees. (A) shows rooted trees $t_1$ and $t_2$. (B) shows the leaves of $t_1$ and $t_2$ identified with each other. (C) shows the gluing: the result of contracting an edge incident to each leaf pairing.}
    \label{fig:gluing}
\end{figure}

\begin{figure}
    \centering
    \includegraphics{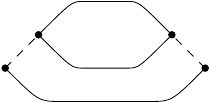}
    \caption{A gluing with a subdivergence from cutting the edges indicated with dashed lines.}
    \label{fig:subdivergence}
\end{figure}

We are interested in how many subdivergence-free gluings a pair of trees has.
\begin{definition}
Given a pair $t_1$ and $t_2$ of rooted trees with the same number of leaves define 
\[
    n(t_1,t_2)
\]
to be the number of subdivergence-free gluings of $t_1$ and $t_2$.
\end{definition}

How many subdivergence-free gluings a pair of trees might have is an interesting question from the perspective of pure enumeration, linking to interesting integer sequences and objects such as connected permutations (see Section~\ref{sec perm}).  We were not aware, during the development and writing of this project, of the previous appearance of these objects as irreducible tanglegrams.  However, the differing motivations lead to different questions on the objects: from the phylogenetic perspective, planarity and number of crossings in the non-planar case is of primary importance \cite{5551100, BLACK2023102550, https://doi.org/10.1002/jgt.22370}, and enumeration questions have generally been studied by running over all trees, while from the quantum field theory perspective, subdivergence-free is of primary importance, and we are interested in enumeration with the input trees fixed.  Consequently, our results are broadly complementary to the existing literature, and also lay out a very different application of these objects.

\subsection{Physics motivation}\label{sec phys mot}

In perturbative quantum field theory one is interested in calculating and understanding the probability amplitudes of particle interactions and other related physical quantities via series expansions.  The most interesting series expansion for us is the Feynman diagram expansion.  In Feynman diagram expansions amplitudes and other quantities are expanded as a series of integrals, called \emph{Feynman integrals}, and each integral corresponds to a graph called the \emph{Feynman graph}, see Figure~\ref{fig:yukawa}.  The Feynman graph can be viewed as an illustration of one possible way the overall interaction could have occurred: the edges correspond to particles and the vertices to interactions between them, external edges correspond to the incoming and outgoing particles of the process we are interested in, while the internal edges correspond to unobservable virtual particles, and the sum is over all possibilities for these.  

\begin{figure}
    \centering
    \includegraphics{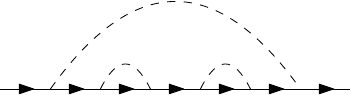}
    \caption{An example of a Feynman graph, in this case the directed edges represent a fermion propagating, while the dashed edges represent mesons.  There are two external edges, one at the left and one at the right.}
    \label{fig:yukawa}
\end{figure}

There are many complexities in this picture that we will not go into at all, but one complexity that is important for the present purposes is that in the interesting cases the Feynman integrals are divergent and so a process known as \emph{renormalization} was developed in order to correct this problem and give a way of obtaining finite, physically correct calculations from these integrals.  Small edge cuts in the graphs (exactly what counts as small depends on the physical theory in question, but typically small is at most three or at most four) result in subgraphs which are themselves divergent, and dealing with these \emph{subdivergences} is an important part of renormalization.  \emph{Propagator subdivergences} are subgraphs which do not contain any external edges but which can be disconnected from the full graph by cutting two edges.\footnote{Bridge edges which upon removal give a component with exactly one external edge also give propagator subdivergences, but by a Legendre transform, see \cite{JKM, Bthesis} for a combinatorial presentation, we may restrict to the bridgeless case and so only consider the kind of propagator subdivergences defined above.}  The subdivergence structure of Feynman graphs gives a Hopf algebra of Feynman graphs and Feynman graphs with no subdivergences are known as \emph{primitive} since they are primitive in the Hopf algebraic sense, see \cite{ck0}.

A \emph{Cutkosky cut} in a Feynman graph is a set of internal edges $C$ of a Feynman graph $\Gamma$ so that removing these edges disconnects $\Gamma$ into $\Gamma_1, \Gamma_2, \ldots, \Gamma_k$ and such that each edge of $C$ has its two ends in different components $\Gamma_i$ (that is, the cut is minimal given the $\Gamma_i$). Additionally we want at least one external edge of the graph in each $\Gamma_i$. Then Cutkosky's theorem gives a formula for calculating the monodromy around certain singularities of the Feynman integral in terms of the $\Gamma_i$ and $\Gamma/\prod\Gamma_i$.  See \cite{BKcut} for a mathematical take on the Cutkosky framework.

In \cite{KYcut}, the third author along with Dirk Kreimer considered series of graphs with Cutkosky cuts and recursive equations building these series.  These equations are similar in form to diagrammatic Dyson-Schwinger equations in quantum field theory, see \cite{Ybook}, which are equations that describe how to build sums of Feynman graphs recursively by inserting, or in parallel how to build the analagous sums of Feynman integrals recursively by an integral equation.  Sticking to the diagrammatic side, we only need to insert into primitive graphs as the subdivergences come from the insertions.  If we want to better understand the monodromy, we need to incorporate Cutkosky cuts into this framework, but the graphs we insert into remain primitive.  In \cite{KYcut}, we work in the \emph{core Hopf algebra} setting, so any bridgless subgraph is divergent, but one could also work in a renormalization Hopf algebra where only small edge cuts lead to divergences, which is what leads to the problem considered in this paper.

Because of this, we are interested in better understanding primitive Feynman graphs with Cutkosky cuts.  The first case to consider is when the Cutkosky cut cuts the graph into exactly two pieces, each with a single external edge.  These cuts tell us about normal thresholds in the Feynman integral.  To further simplify we may additionally restrict to the case where the Cutkosky cut cuts every cycle of the graph so that the components after cutting are trees.  These cuts correspond to almost-leading singularities (see Remark 4.1 and Section 5.6 of \cite{KYcut}).  If we reverse the problem then we are considering gluing two trees along the Cutkosky cut.  The original external edges become the roots of the two trees while the cut edges become the leaves.  Our original graph was primitive so the gluing should not have subdivergences.  Because the original graph had only two external edges, it is actually most interesting to only forbid propagator subdivergences.  Additionally the vertices in our Feynman graphs will have degree at least 3 and so any propagator subdivergence from gluing two trees will cross the Cutkosky cut.  This brings us back to the combinatorial problem we began with --- we wish to understand (propagator) subdivergence-free gluings of trees.

\subsection{Initial observations}\label{sec initial}


We note a few preliminary observations that help us understand the problem of counting subdivergence-free gluings of trees. First, consider the most basic family of trees, fans, that is trees consisting of a root and a positive number of leaves all of which are children of the root. If either tree $t_1$ or $t_2$ is a fan with $j$ leaves, $n(t_1,t_2) = j!$ since the fan has no non-leaf edges and so no cuts of the required type to give subdivergences can exist and so any identification of the leaves of $t_1$ with the leaves of $t_2$ gives a subdivergence-free gluing. 


If $t_1$ and $t_2$ are rooted trees, and $e$ is an edge of $t_1$, then $n(t_1 / e, t_2) \geq n(t_1, t_2)$. This is because every subdivergence-free gluing of $t_1/e$ and $t_2$ gives a subdivergnce-free gluing of $t_1$ and $t_2$ by identifying the leaves in the same way and cutting the same edges, but the converse does not hold as the cut may use edge $e$.

However, whenever we have a 2-valent vertex in one of the trees, contracting an edge will not change the number of subdivergence-free gluings, since cuts using both edges incident to the 2-valent vertex are forbidden as they do not use an edge from each tree, and if a cut using one of the edges incident to the 2-valent vertex defines a subdivergence then it remains so after swapping which edge incident to the 2-valent vertex is used.

Given a pair of trees, one could count the number of subdivergence-free gluings by brute force; simply check each permutation for subdivergences.  We will present better algorithms in Section~\ref{sec algs}. 

\subsection{Outline}

The paper will approach counting subdivergence-free gluings in two ways. The first, described in Sections~\ref{sec perm} and ~\ref{sec families}, explores closed-form expressions to count subdivergence-free gluings of specific infinite families of tree pairings. We present useful lemmas and definitions in Section~\ref{sec perm}, and derive closed-form expressions for three different families of trees in Section~\ref{sec families}. The second approach focuses on algorithms to count subdivergence-free gluings between arbitrary pairs of trees, and is covered in Section~\ref{sec algs}. We present two recursive algorithms, one based on dealing recursively with the children of the root and the other based on the techniques used to count the specific families in Section~\ref{sec families}. Section~\ref{sec discussion} discusses potential future directions and relate our results back to questions in quantum field theory.

\section{Connected permutations and generalizations}\label{sec perm}

In this section we present a number of definitions that will be useful for our counting in Section~\ref{sec families}. These definitions also make explicit the connection between subdivergence-free gluings of trees and established classes of permutations.

Permutations of $\{1,\dots, n\}$ that do not map any prefix of $\{1,\dots,n\}$ to itself have been a standard example used in enumeration since at least Comtet \cite{Cbook}.  See \cite{Kperm} for a list of many such appearances.  These permutations are known as connected permutations, irreducible permutations, or indecomposable permutations and appear as A003319 in \cite{connected_permutations_oeis}.  Specifically, we can define them as follows.

\begin{definition}
A permutation $\sigma$ of $\{1,\dots,n\}$ is said to \emph{fix a prefix of size $j$} if \[\sigma(\{1,\dots,j\}) = \{1,\dots,j\}\].
\end{definition}

\begin{definition}
A \emph{connected permutation of size $n$} is a permutation $\sigma$ of $\{1,\dots,n\}$ such that for all $j \in \{1,\dots,n-1\}$, $\sigma$ does not fix a prefix of length $j$. We denote the number of connected permutations of size $n$ by $c_n$.
\end{definition}

The following expression for $c_n$ is well-known, but we give a proof as a warm-up to the arguments of Section~\ref{ssec preproc} as their structure is similar.
\begin{lemma}\label{lem connectedpermutation}
$c_n = n! - \sum_{i=1}^{n-1} i! c_{n-i}$
\end{lemma}

\begin{proof}
From both the definition of $c_n$ and the expression above we have that $c_1 = 1$. 

For $n > 1$, we take the $n!$ permutations of $\{1,\dots,n\}$, and subtract the number of permutations that fix prefixes. We count permutations that fix a prefix by summing the number of permutations that fix a prefix of length $i$ and no longer, for each value of $i$ from $1$ through $n-1$. We count permutations that fix prefixes of length $i$ and no longer by multiplying the $i!$ ways to fix $\{1,\dots,i\}$ by the $c_{n-i}$ ways to map the rest of the elements so that overall no prefix longer than $i$ is fixed.
\end{proof}

We will find it useful to extend the notion of connected permutation to a more general definition.

\begin{definition}
Let $S \subseteq \{1,\dots,n-1\}$. An \emph{S-connected permutation of size $n$} is a permutation $\sigma$ of $\{1,\dots,n\}$ such that for all $j \in S$, $\sigma$ does not fix a prefix of length $j$. Let $c_n^S$ denote the number of S-connected permutations of size $n$.
\end{definition}

\begin{lemma}\label{lem sconnectedpermutation}
$c_n^S = n! - \sum_{i \in S} i! c_{n-i}^{\{1,\dots,n-i-1\} \cap \{s_1-i,\dots,s_k-i\}}$, where $S = \{s_1,\dots,s_k\}$. 
\end{lemma}

\begin{proof}
We count these by first taking all $n!$ permutations and then subtracting permutations that fix prefixes in $S$. We count the latter by summing, for each $i \in S$, the number of permutations that fix a prefix of length $i$ but do not fix a prefix in $S$ of length greater than $i$. For a fixed $i \in S$, the number of such permutations is given by multiplying the $i!$ ways to fix a prefix of length $i$ by the $c_{n-i}^{\{1,\dots,n-i-1\} \cap \{s_1-i,\dots,s_k-i\}}$ ways to map the rest of the elements so that no prefix in $S$ of length greater than $i$ is fixed.  Each permutation which is not S-connected is subtracted exactly once, namely in the term indexed by the maximal $i\in S$ for which the prefix of length $i$ is fixed.
\end{proof}

Note that when $S = \{1,\dots,n-1\}$, $c_n^S = c_n$, and for each $i \in S$, $c_{n-i}^{\{1,\dots,n-i-1\} \cap \{s_1-i,\dots,s_k-i\}} = c_{n-i}$. Hence Lemma ~\ref{lem connectedpermutation} is the special case of Lemma ~\ref{lem sconnectedpermutation} with $S = \{1,\dots,n-1\}$.

\section{Counting subdivergence-free gluings in special families}\label{sec families}

In this section, we count subdivergence-free gluings between trees in specific infinite families. We examine three different families. For each family, we will be able to use the definitions and results in Section ~\ref{sec perm} to derive closed-form expressions for the number of subdivergence-free gluings. The first family illustrates this connection to connected permutations explicitly. The second family introduces a root with multiple children. In this case there are multiple subsets of edges that can be cut to generate a subdivergence for a given gluing. The third family has an internal vertex with no leaves attached. Here there are multiple subsets of edges along the same branch that can generate a subdivergence for a given gluing.

A vertex of a rooted tree which is not a leaf is called an \emph{internal vertex} and an edge which is not incident to a leaf is called an \emph{internal edge}.

\subsection{Line Family}

We first consider families of trees where each node in the tree has at most one non-leaf child, or equivalently where the internal nodes form a path, beginning with the following special case.

\begin{definition}\label{def first_tree_family}
Let $\ell_k$ be the tree with $k$ internal vertices where the internal vertices form a path, the root is at one end, and there is one leaf attached to each internal vertex. See Figure ~\ref{fig:first_tree_family} for examples.
\end{definition}

\begin{figure}
\begin{subfigure}{.3\textwidth}
  \centering
  \includegraphics[]{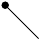}
  \caption{$\ell_1$}
  \label{fig:t_1}
\end{subfigure}%
\begin{subfigure}{.3\textwidth}
  \centering
  \includegraphics[]{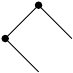}
  \caption{$\ell_2$}
  \label{fig:t_2}
\end{subfigure}%
\begin{subfigure}{.3\textwidth}
  \centering
  \includegraphics[]{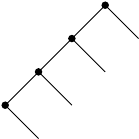}
  \caption{$\ell_4$}
  \label{fig:t_4}
\end{subfigure}
    \caption{Examples of trees $\ell_k$ in the line family.}
    \label{fig:first_tree_family}
\end{figure}

\begin{thm}\label{thm first_tree_family}
The number of subdivergence-free gluings of $\ell_k$ with itself is $n(\ell_k,\ell_k) = c_k$.
\end{thm}

\begin{proof}
Label the leaves of each $\ell_k$ from $1$ to $k$ by labelling the deepest leaf $1$ and then continuing up the tree, as in Figure ~\ref{fig:first_tree_family_labelled}. Then the subdivergence-free condition says that the permutation giving the gluing cannot fix any prefix of the labels of length $1$ through $k-1$.
\end{proof}

\begin{figure}
    \centering
    \includegraphics[]{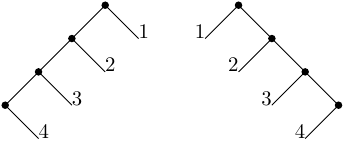}
    \caption{Labelling of $\ell_4$. Any gluing that fixes a prefix of length 1,2 or 3 creates a subdivergence.}
    \label{fig:first_tree_family_labelled}
\end{figure}

\begin{definition}\label{def complex_line_fam}
Let $S = \{s_1,\dots,s_j\} \subseteq \{1,\dots,k-1\}$, with $s_1 \leq s_2 \leq \dots s_j$. 
Let $\ell_{k,S}$ be the tree with $j+1$ internal vertices, where the internal vertices form a path with the root at one end, and where the $j-i$th internal vertex from the root has $s_{i+1}-s_i$ leaves, with the convention that $s_0=0$ and $s_{j+1} = k$.
See Figure ~\ref{fig:complex_line_family} for examples.
\end{definition}

The tree $\ell_{k,S}$ can also be built by the following recursive process. Start with the vertex with $s_1$ leaves, and call this vertex $1$. Recursively attach parent vertices with enough leaves that vertex $i$ has $s_i$ leaves below it, up to vertex $j$. Finally, attach the root vertex with enough leaves that the entire tree has $k$ leaves. 

\begin{figure}
\begin{subfigure}{.5\textwidth}
  \centering
  \includegraphics[]{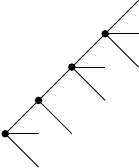}
  \caption{$\ell_{8,\{3,5,6\}}$}
  \label{fig:complex_line_fam}
\end{subfigure}%
\begin{subfigure}{.5\textwidth}
  \centering
  \includegraphics[]{first_tree_family2.pdf}
  \caption{$\ell_{4,\{1,2,3\}}$}
  \label{fig:complex_line_fam_basic}
\end{subfigure}
    \caption{More general examples of trees in the line family.}
    \label{fig:complex_line_family}
\end{figure}

\begin{thm}\label{thm complex_line_fam}
The number of subdivergence-free gluings of $\ell_{k, S_1}$ with $\ell_{k, S_2}$ is $n(\ell_{k,S_1},\ell_{k,S_2}) = c_{k}^{S_1 \cap S_2}$.
\end{thm}

\begin{proof}
Say $S_1 = \{s_1,s_2,\dots,s_j\}$ and $S_2 = \{q_1,q_2,\dots,q_i\}$. Label the leaves of $t_1=\ell_{k, S_1}$ from $1$ to $k$ by labelling the fan of $s_1$ leaves of the internal vertex farthest from the root from $1$ to $s_1$, and then continuing along the tree so that the vertex with $s_i$ leaves below it also has labels $1$ through $s_i$ below it. Label $t_2=\ell_{k, S_2}$ similarly. Then the subdivergence-free condition says that the permutation giving the gluing can not fix any prefix of the labels whose length is in $S_1 \cap S_2$. This is because the only way to get a subdivergence is to have all of the leaves below a vertex in $t_1$ joined to all of the leaves below a vertex in $t_2$, which is equivalent to fixing a prefix in $S_1 \cap S_2$.
\end{proof}

\subsection{Two-ended family}

\begin{definition}
Let $d_{(i,j)}$ be the tree formed by a vertex whose two children are $\ell_i$ and $\ell_j$ in the sense of Definition ~\ref{def first_tree_family}. See Figure ~\ref{fig:two_ended_family} for examples.
\end{definition}

\begin{figure}
\begin{subfigure}{.5\textwidth}
  \centering
  \includegraphics[]{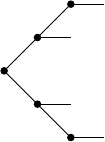}
  \caption{$d_{(2,2)}$}
  \label{fig:double_t22}
\end{subfigure}%
\begin{subfigure}{.5\textwidth}
  \centering
  \includegraphics[]{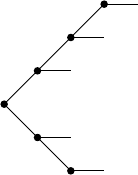}
  \caption{$d_{(3,2)}$}
  \label{fig:double_t32}
\end{subfigure}
    \caption{Examples of trees in the two-ended family.}
    \label{fig:two_ended_family}
\end{figure}

\begin{thm}\label{thm two_sided_same_length} 
The number of subdivergence-free gluings of $d_{(k, k)}$ with itself is
\[
n(d_{(k,k)},d_{(k,k)}) = (2k)! - 2\bigg(k!k! + \sum\limits_{\substack{1 \leq i \leq k-1 \\ 1 \leq j \leq k-1}} i!j!c_{2k-i-j} + 2(\sum_{i = 1}^{k-1} i!c_{2k-i})\bigg).
\]
\end{thm}

\begin{proof}
Order the leaves of the trees beginning at the end of one $\ell_k$ and ending at the end of the other $\ell_k$.

There are $(2k)!$ gluings in total. We subtract gluings that result in subdivergences. A simple way to get a subdivergence is to fix the first $k$ leaves of each tree together and the last $k$ leaves of each tree together. There are $k!k!$ ways to do this. Likewise, there are $k!k!$ ways of gluing the first $k$ leaves of the first to the last $k$ leaves of the second tree and the last $k$ leaves of the first tree to the first $k$ leaves of the second tree. This accounts for all gluings with subdivergences resulting from cutting either pair of edges nearest the root.

Next we count ways to fix the first $k-1$ leaves of the two trees together and last $k-1$ leaves of the two trees together and fix no prefix of the middle two leaves.  There are $(k-1)!(k-1)!c_2$ ways to do this. None of these have been already counted, since cutting the edges nearest the root would not result in a subdivergence. Then we count ways to fix the first $k-1$ leaves of the two trees together and the last $k-2$ leaves of the two trees together, and fix no prefix of the middle three leaves (If a prefix of length one is fixed, so is a suffix of length two and we are in the k!k! case. If a prefix of length two is fixed, so is a suffix of length one and we are in the (k-1)!(k-1)! case). There are $(k-1)!(k-2)!c_3$ ways to do this.  Similarly, there are $i!j!c_{2k-i-j}$ ways to glue the first $i$ leaves of the two trees together and the last $j$ leaves of the two trees together, but no larger prefix on either side, contributing \[\sum\limits_{\substack{1 \leq i \leq k-1 \\ 1 \leq j \leq k-1}} i!j!c_{2k-i-j}\] gluings with subdivergences.  We get the same number of subdivergent gluings where a prefix of the first tree is glued to a suffix of the second tree and a suffix of the second tree is glued to a prefix of the first tree.
We have now accounted for all gluings with two possible places to generate subdivergences, one on each branch of the tree.


We also count ways to have only one subdivergence cut. There are $k-1$ places to cut on each branch. If we cut after leaf $i$ for $1 \leq i \leq k-1$, the first $i$ leaves can be fixed in any way, and the remaining leaves can be fixed so long as they do not allow a subdivergence, which happens so long as they do not fix any prefixes. Overall, this contributes $\sum_{i = 1}^{k-1} i!c_{2k-i}$ gluings with subdivergences. We can do the same with gluing the last leaves to themselves, generating the same number of trees with subdivergences, and likewise gluing a prefix of the first tree to a suffix of the second tree or vice versa.  This gives the number in the theorem statement.

\end{proof}

\begin{thm}
When $k \neq l$, the number of subdivergence-free gluings of $d_{(k,l)}$ with itself is
\[
n(d_{(k,l)},d_{(k,l)}) = (k+l)! - A - B
\]
where 
\begin{align*}
    A =& k!l! + \sum\limits_{\substack{1 \leq i \leq k-1 \\ 1 \leq j \leq l-1}} i!j!c_{k+l-i-j} + \sum_{i=1}^{k-1}i!c_{k+l-i} + \sum_{j=1}^{l-1}j!c_{k+l-j}, \\
    B =& \sum\limits_{\substack{1 \leq i \leq a \\ 1 \leq j \leq a}}i!j!c_{k+l-i-j}^{S_1} + 2\sum_{i=1}^{a}i!c_{k+l-i}^{S_2},\\
    a =& \min(k,l), \\
    S_1 =& \{1,\dots,a-i\}\cup\{k+l-i-j-(a-j),\dots,k+l-i-j-1\}, and \\
    S_2 =& \{1,\dots,a-i\}\cup\{k+l-i-a,\dots,k+l-i-1\}
\end{align*}    
\end{thm}

\begin{proof}
The expression $A$ is derived in the same way as in the proof of Theorem ~\ref{thm two_sided_same_length}, looking at subdivergences arising from gluing the end of the $k$ leaf branches of both trees together or the $l$ leaf branches of both trees together. The expression $B$ is derived by considering subdivergences arising from gluing the ends of the $k$ leaf branches to the ends of the $l$ leaf branches, as illustrated in Figure ~\ref{fig:double_lopsided_2}. 

We first count gluings with two possible ways to generate subdivergences. The simplest is to fix the first $a := \min(k,l)$ leaves together and last $a$ leaves together, with no smaller subdivergence within either of these parts of the gluing, and allowing the middle leaves to be glued in any order. This amounts to $a!a!(k+l-2a)!$, or $a!a!c_{k+l-2a}^{\emptyset}$ gluings. Next we consider fixing the first $a-1$ leaves together, and the last $a$ leaves together. The middle edges must be glued without fixing a prefix of length 1, as this has already been counted in the $a!a!$ case just considered. Thus there are $(a-1)!a!c_{k+l-2a+1}^{\{1\}}$ gluings. Fixing the first $a$ leaves together and the last $a-1$ leaves is similar, but now we must avoid fixing a suffix of length 1, giving $a!(a-1)!c_{k+l-2a+1}^{\{k+l-2a\}}$ gluings. This continues down to having just the first leaf fixed to itself and the last leaf fixed to itself. This generates $\sum\limits_{\substack{1 \leq i \leq a \\ 1 \leq j \leq a}}i!j!c_{k+l-i-j}^{S_1}$ gluings, with $S_1$ as in the theorem statement.

We also count ways with a subdivergence on only one side. We first consider ways when the top ends of the trees are glued together, and the other case is symmetric. There are $a$ places to cut, which result from gluing the first $i$ leaves together for $1 \leq i \leq a$. So long as the trees have no suffix of length $1$ through $a$ fixed, there will be no subdivergences on the other branches of the trees. As well, when we are counting gluings that fix the first $i$ leaves, if we count gluings that do not fix prefixes of the remaining leaves of length $1$ through $a-i$, we will only count each gluing once. This generates $2\sum_{i=1}^{a}i!c_{k+l-i}^{S_2}$ gluings, with $S_2$ as in the theorem statement.
\end{proof}

\begin{figure}
\begin{subfigure}{.5\textwidth}
  \centering
  \includegraphics[]{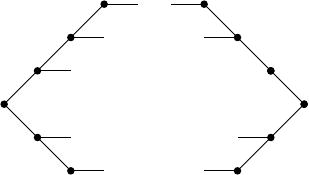}
  \caption{}
  \label{fig:double_lopsided_1}
\end{subfigure}%
\begin{subfigure}{.5\textwidth}
  \centering
  \includegraphics[]{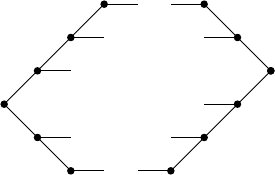}
  \caption{}
  \label{fig:double_lopsided_2}
\end{subfigure}
    \caption{Two ways of gluing to get subdivergences.}
    \label{fig:double_lopsided_family}
\end{figure}

\subsection{Lines with an extra fan family}

\begin{definition}
For $1 < i < k$, let $f_{k,i}$ be the tree with $k$ leaves formed in the same way as $\ell_k$ in Definition ~\ref{def first_tree_family}, but where the leaf edge at the $i$th deepest vertex is replaced by an edge connected to a vertex with a single leaf as a child. We will refer to this added internal edge as the \emph{extra edge}. See Figure ~\ref{fig:third_tree_family} for an example.
\end{definition}

\begin{figure}
  \centering
  \includegraphics[]{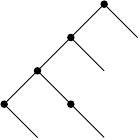}
  \caption{Example of $f_{4,3}$}
  \label{fig:third_tree_family}
\end{figure}

\begin{thm}\label{thm third_tree_family}
The number of subdivergence-free gluings of $f_{k,i}$ with itself is
$n(f_{k,i},f_{k,i}) = c_k - 2c_{k-1}^{\{i-1,\dots,k-2\}} + c_{k-2}^{\{i-2,\dots,k-3\}} - c_{k-1}.$
\end{thm}

\begin{figure}
\begin{subfigure}{.5\textwidth}
  \centering
  \includegraphics[]{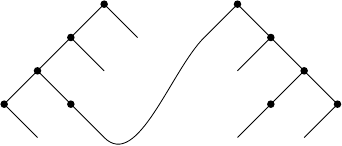}
  \caption{}
  \label{fig:clair_first_counting_case}
\end{subfigure}%
\begin{subfigure}{.5\textwidth}
  \centering
  \includegraphics[]{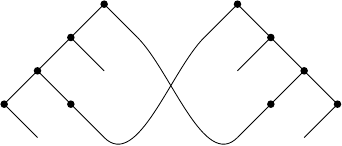}
  \caption{}
  \label{fig:clair_second_counting_case}
\end{subfigure}%

\begin{subfigure}{.3\textwidth}
  \centering
  \includegraphics[]{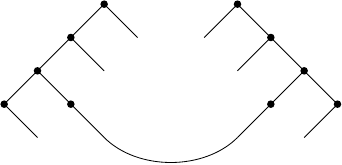}
  \caption{}
  \label{fig:clair_third_counting_case}
\end{subfigure}%
    \caption{Inclusion and exclusion cases for counting third tree family.}
    \label{fig:counting_third_tree_family}
\end{figure}

\begin{proof}
We count $n(f_{k,i},f_{k,i})$ with an inclusion-exclusion approach, contracting the extra edge in each tree to simplify the counting. We will refer to the two copies of $f_{k,i}$ as $t_1$ and $t_2$, the extra edges of $t_1$ and $t_2$ as $e_1$ and $e_2$, respectively, and the trees formed from contracting the extra edges as $t_1 / e_1$ and $t_2 / e_2$. We will first consider $n(t_1 / e_1, t_2 / e_2)$. Since $t_1 / e_1$ and $t_2 / e_2$ are both just $l_k$ in the sense of Definition ~\ref{def first_tree_family}, this amounts to $c_k$ by Theorem ~\ref{thm first_tree_family}. $n(t_1 / e_1, t_2 / e_2)$ over counts $n(t_{k,i},t_{k,i})$ as it does not consider subdivergences from cutting $e_1$ or $e_2$. Such subdivergences can occur either when the leaf after $e_1$ is connected to the end of $t_2$, the leaf after $e_2$ is connected to the end of $t_1$, or the leaf after $e_2$ is connected to the leaf after $e_1$. Thus we must subtract the number of gluings counted in $n(t_1 / e_1, t_2 / e_2)$ that contain these subdivergences.

We first count the number of gluings with subdivergences from either joining the leaf after $e_1$ to the end of $t_2$, the leaf after $e_2$ to the end of $t_1$, or both. If the leaf after $e_1$ is joined to the end of $t_2$, as in Figure ~\ref{fig:clair_first_counting_case}, there are $k-1$ leaves left on each tree, and any gluing of these that fixes a prefix of length $i-1$ or more would have already been counted as a subdivergence in $n(t_1 / e_1, t_2 / e_2)$. There are thus $c_{k-1}^{\{i-1,\dots,k-2\}}$ gluings with subdivergences from joining the leaf after $e_1$ to the end of $t_2$ which have not already been accounted for. Similarly, there are $c_{k-1}^{\{i-1,\dots,k-2\}}$ gluings with subdivergences from joining the leaf after $e_2$ to the end of $t_1$. However, this double counts the gluings with subdivergences from joining both the leaf after $e_1$ to the end of $t_2$ and the leaf after $e_2$ to the end of $t_1$, as in Figure ~\ref{fig:clair_second_counting_case}, so we must subtract such gluings. Assuming that these leaves are joined, there are $k-2$ leaves left on each tree, and any gluing of these that fixes a prefix of length $i-2$ or more would have already been counted as a subdivergence in $n(t_1 / e_1, t_2 / e_2)$. Thus we must subtract $c_{k-2}^{\{i-2,\dots,k-3\}}$ gluings. In summary, the total number of gluings to subtract from $n(t_1 / e_1, t_2 / e_2)$ to account for subdivergences from joining the leaf after $e_1$ to the end of $t_2$, the leaf after $e_2$ to the end of $t_1$, or both, is $2c_{k-1}^{\{i-1,\dots,k-2\}} - c_{k-2}^{\{i-2,\dots,k-3\}}$.

Finally, the number of gluings to remove with subdivergences from joining the leaf after $e_1$ to the leaf after $e_2$ is just those that join these leaves but do not have any other subdivergences, as in Figure ~\ref{fig:clair_third_counting_case}. There are $k-1$ leaves remaining, forming a structure like $l_{k-1}$ in the sense of Definition ~\ref{def first_tree_family}, so the number of gluings here is $c_{k-1}$ by Theorem ~\ref{thm first_tree_family}.
\end{proof}

\begin{definition}
For $1 < i < k$ and $j \geq 1$, let $f_{k,i,j}$ be the tree formed by starting with $f_{k,i}$ and adding $j-1$ more leaves to both the deepest internal vertex, and to the vertex below the extra edge. See Figure ~\ref{fig:bigfans} for an example.
\end{definition}

\begin{figure}
  \centering
  \includegraphics[]{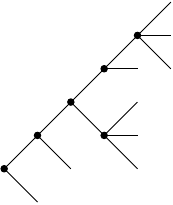}
  \caption{Example of $f_{5,3,3}$}
  \label{fig:bigfans}
\end{figure}

\begin{thm}\label{thm bigfans}
The number of subdivergence-free gluings of $f_{k,i,j}$ with itself is
\begin{align*}
    n(f_{k,i,j},f_{k,i,j}) =& c_{k+2(j-1)}^{\{j,\dots,i+j-2,i+2(j-1),\dots,k+2(j-1)-1\}}
- 2j!c_{k+j-2}^{\{i+j-2,\dots,k+j-3\}} \\
& + j!j!c_{k-2}^{\{i-2,\dots,k-3\}}
- j!c_{k+j-2}^{\{j,\dots,k+j-3\}}.
\end{align*}
\end{thm}

\begin{proof}
We use a similar inclusion-exclusion approach to the proof of Theorem ~\ref{thm third_tree_family}. Again, we will refer to the two copies of $f_{k,i,j}$ as $t_1$ and $t_2$, the extra edges of $t_1$ and $t_2$ as $e_1$ and $e_2$, respectively, and the trees formed from contracting the extra edges as $t_1/e_1$ and $t_2/e_2$. Again, we first calculate $n(t_1/e_1,t_2/e_2)$, and subtract gluings that would result in subdivergences from cutting $e_1$ or $e_2$. Since $t_1/e_1$ and $t_2/e_2$ are both just $l_{k+2(j-1),\{j,\dots,i+j-2\,i+2(j-1),\dots,k+2(j-1)-1\}}$ in the sense of Definition ~\ref{def complex_line_fam}, it follows by Theorem ~\ref{thm complex_line_fam} that \\ $n(t_1/e_1,t_2/e_2) = c_{k+2(j-1)}^{\{j,\dots,i-1+j-1,i+2(j-1),\dots,k+2(j-1)-1\}}$.

Now we count the number of gluings with subdivergences from joining the fan after $e_1$ to the fan at the end of $t_2$. If we assume that these fans are joined, then there are $k-1+j-1$ leaves remaining, and any gluing of these that fixes a prefix of length $i+j-2$ or more would have already been counted as a subdivergence in $n(t_1/e_1,t_2/e_2)$. Thus there are $j!$ ways to join the fans, and $c_{k+j-2}^{\{i+j-2,\dots,k+j-3\}}$ ways to join the remaining leaves so that there is no subdivergence. There are the same number of gluings with subdivergences from joining the fan after $e_2$ to the end of $t_1$. This double counts gluings with subdivergences from joining both the fan after $e_1$ to the end of $t_2$ and the fan after $e_2$ to the end of $t_1$, so we must subtract them. Assume these fans are joined in any of the $j!j!$ possible ways. There are then $k-2$ remaining leaves on each tree, and any gluing of these that fixes a prefix of length $i-2$ or larger would have already been counted as a subdivergence in $n(t_1/e_1,t_2/e_2)$. Thus there are $j!j!c_{k-2}^{\{i-2,\dots,k-3\}}$ gluings that we have double counted here. In summary, the total number of gluings to subtract from $n(t_1/e_1,t_2/e_2)$ to account for subdivergences from joining the fan after $e_1$ to the end of $t_2$, the fan after $e_2$ to the end of $t_1$, or both, is $2j!c_{k+j-2}^{\{i+j-2,\dots,k+j-3\}} - j!j!c_{k-2}^{\{i-2,\dots,k-3\}}$.

Finally, the number of gluings to remove with subdivergences from joining the fan after $e_1$ to the fan after $e_2$ is just those that join these fans but do not have any other subdivergences. If these fans are joined in any of the $j!$ possible ways, there are $k+j-2$ leaves remaining on each tree, that form a structure like $l_{k+j-2,\{j,\dots,k+j-3\}}$ in the sense of Definition ~\ref{def complex_line_fam}, so the number of gluings here is $c_{k+j-2}^{\{j,\dots,k+j-3\}}$ by Theorem ~\ref{thm complex_line_fam}. So once we remove these $j!c_{k+j-2}^{\{j,\dots,k+j-3\}}$ over counted gluings from what we had thus far, we arrive at $n(t_1,t_2)$.
\end{proof}

Note that Theorem ~\ref{thm third_tree_family} is a special case of Theorem ~\ref{thm bigfans}, in the case where $j=1$.



\section{Algorithms}\label{sec algs}

This section presents two algorithms for computing subdivergence-free gluings of arbitrary pairs of trees.  

\subsection{Eating away children of the root}

The first algorithm will be presented in the form of recursive formulas for $n(t_1, t_2)$ which inductively determine $n(t_1, t_2)$ for any pair of rooted trees $t_1$ and $t_2$.  These formulas directly determine a recursive algorithm given by simply following the formulas.  A Sage implementation of this algorithm is included with the source of the arXiv version of this paper \cite{this}, but the strength of this algorithm is the recursive reformulation of $n(t_1, t_2)$, more than the algorithm itself on account of its performance.

For this algorithm we will need a variant on the notation of a subdivergence.  For this subsection only what we have previously called a subdivergence in a gluing of two trees will be called a \emph{fully internal subdivergence}, while the new notion will be known as a \emph{one-sided subdivergence}.  
The motivation for this language is that for subdivergences as we have dealt with them so far, we cut two internal edges in order to cut out a subdivergence that contains neither root.  That is, speaking physically, the subdivergence has none of the external edges of the glued graph, rather the external edges of the subdivergence both result from the cut edges.  These subdivergences are in this sense \emph{fully internal}.  As discussed in Section~\ref{sec phys mot}, it is also possible from a physical perspective to have a subdivergence where one external edge is an external edge of the original graph and the other results from a cut edge, and these we call \emph{one-sided}.  One-sided propagator subdivergences only occur when the graph has a bridge, and by a Legendre transform we can ignore this situation in general, but for the purposes of this algorithm, we will have partial gluings of trees, and these partial gluings may result in bridges and hence in one-sided subdivergences.  Note that unglued leaves also count as external edges and so are not allowed in the one-sided subdivergences.  Formal definitions are given in what follows.

\medskip

Let $t_1$ and $t_2$ be rooted trees as above, but not necessarily with the same number of leaves.  Let $l(t)$ be the number of leaves in the rooted tree $t$ and $L(t)$ the set of leaves of $t$.  We will need to keep track of partial gluings of leaves of $t_1$ and $t_2$, so given a function $\tau$ which is a bijection between a subset of the leaves of $t_1$ and a subset of leaves of $t_2$, we can form the graph $g_\tau(t_1, t_2)$ by identifying leaves according to $\tau$ in the same way as for the tree gluings we have been working with so far.  Call $\tau$ a \emph{partial gluing} and let the size of $\tau$ be the size of the image of $\tau$.

There is a \emph{fully internal subdivergence} of $g_\tau(t_1, t_2)$ if there are two edges $e_1$ and $e_2$, where $e_i\in t_i$ and neither $e_i$ is incident to a leaf of $t_i$, and where $e_1, e_2$ is a 2-edge cut of $g_\tau(t_1, t_2)$ such that one of the parts of the cut contains neither of the roots nor any of the unpaired leaves.  There is a \emph{one-sided subdivergence} of $g_\tau(t_1, t_2)$ if there is an edge $e$ which is a bridge of $g_\tau(t_1, t_2)$ with the property that one of the parts has none of the unpaired leaves but has the root from the tree that did not contain $e$.  This subdivergence is \emph{left-sided} if the part with no unpaired leaves has the root of $t_1$ in it and \emph{right-sided} otherwise.

For $S_i \subseteq L(t_i)$, define
\[
p_{k,S_1,S_2}(t_1, t_2)
\]
to be the number of partial gluings of size $k$ with domain a subset of $S_1$ and image a subset of $S_2$ which have no subdivergences (either fully internal or one-sided).  Define
\[
\overline{p}_{k,S_1,S_2}(t_1, t_2)
\]
to be the number of partial gluings of size $k$ with domain a subset of $S_1$ and image a subset of $S_2$ which have no fully internal or right-sided subdivergences.

When $l(t_1)=l(t_2)$ note that 
\[
n(t_1, t_2) = p_{l(t_1), L(t_1), L(t_2)}(t_1, t_2),
\]
so if we can calculate all $p_{k,S_1, s_2}(t_1, t_2)$ then we can calculate all $n(t_1, t_2)$. 

We have the symmetry $p_{k, S_1, S_2}(t_1, t_2) = p_{k, S_2, S_1}(t_2, t_1)$.
Note however, that $\overline{p}$ is not symmetric in $1,2$.  An analogous object could be defined with no fully internal or left-sided subdivergences, but we will not need this.

\medskip

Let $f_k$ be the rooted tree consisting of a root with $k$ leaves as children, so a rooted $k$-corolla, or $k$-fan, or $\ell_{k, \emptyset}$ in the notation of Section~\ref{sec families}.

The following observations follow from the definitions; the second point can be taken as a definition of how the one vertex tree behaves.
\begin{itemize}
    \item If $k>l(t_1)$ or $k>l(t_2)$ or $|S_1|<k$ or $|S_2|<k$ then 
    \[
    p_{k, S_1, S_2}(t_1, t_2) = \overline{p}_{k, S_1, S_2}(t_1, t_2) = 0.
    \]
    \item $p_{1, \{\bullet\}, S_2}(\bullet, t_2) = \overline{p}_{1, \{\bullet\}, S_2}(\bullet, t_2) = |S_2|$, where $\bullet$ is the rooted tree with a single vertex which is both root and leaf.
    \item 
    \[
    p_{k, S_1, S_2}(f_i, f_j) =  \overline{p}_{k, S_1, S_2}(f_i, f_j) = k!\binom{|S_1|}{k}\binom{|S_2|}{k} 
    \]
    \item If $k \neq l(t_1)$ then $p_{k, S_1, S_2}(t_1, t_2) = \overline{p}_{k, S_1, S_2}(t_1, t_2)$.
\end{itemize}
The last point is true because the only way to have a left-sided subdivergence is when $S_1=L(t_1)$ and so when that is impossble, there can be no distinction between $p$ and $\overline{p}$.  We also see from this that $p$ and $\overline{p}$ are not very different, but their difference will be key to making the recursion work in all cases.

\begin{definition}
Let $t_1, t_2, \ldots, t_k$ be rooted trees.  Define $B_+(t_1,t_2,\cdots ,t_k)$ to be the rooted tree formed from $t_1, t_2, \ldots, t_k$ by adding a new root vertex whose children are the roots of the $t_i$. 
\end{definition}

 We will first give a formula for $p$ and $\overline{p}$ for trees where the root has one child.

\begin{prop}\label{prop 1 child}
If $t_1 = B_+(t_{11})$ then
\[
  \overline{p}_{k, S_1, S_2}(t_1, t_2) = p_{k, S_1, S_2}(t_1, t_2)
  = \begin{cases}
    p_{k, S_1, S_2}(t_{11}, t_2) & \text{if $k\neq l(t_1)$ or $k\neq l(t_2)$}\\
    0 & \text{if $k=l(t_1) = l(t_2)$.}
  \end{cases} 
\]
\end{prop}

\begin{proof}
The first equality holds because any partial gluing giving a left-sided subdivergence also gives a fully internal subdivergence by cutting the edge incident to the root of $t_1$ along with the cut edge from the left-sided subdivergence.  

For the second equality, first note that if $k=l(t_1)=l(t_2)$ then all leaves must be glued, and so the edge incident to the root is a bridge giving a right-sided subdivergence for any gluing, hence $\overline{p}_{k, S_1, S_2}(t_1, t_2) = p_{k, S_1, S_2}(t_1, t_2) = 0$
If $k\neq l(t_1)$ or $k\neq l(t_2)$ then not all leaves are glued and so any subdivergence of such a partial gluing of $t_1$ and $t_2$ includes at most one of the root of $t_2$ and the root of $t_{11}$ and hence is a subdivergence (possibly one-sided) of the same partial gluing of $t_{11}$ and $t_2$ and conversely any subdivergence in a partial gluing of $t_{11}$ $t_{2}$ remains a subdivergence in the same partial gluing of $t_1$ and $t_2$.
\end{proof}

\begin{prop}\label{prop 2 children}
Suppose $t_1 = B_+(t_{11}, t_{12}, \ldots)$ with at least two subtrees.  Let $\overline{t} = B_+(t_{12}, \ldots)$ and let $S_{11} = S_1 \cap L(t_{11})$. Then
\[
\overline{p}_{k, S_1, S_2}(t_1, t_2) = \sum_{\substack{0 \leq j \leq k\\R\subseteq S_2\\|R|=j}} p_{j, S_{11}, R}(B_+(t_{11}), t_2)\overline{p}_{k-j, S_1 - S_{11}, S_2-R}(\overline{t}_1, t_2)
\]
and
\[
p_{k, S_1, S_2}(t_1, t_2)
= \begin{cases}
  \overline{p}_{k, S_1, S_2}(t_1, t_2) & \text{if $l(t_1)\neq k$} \\
  \displaystyle\sum_{\substack{R\subseteq S_2\\|R|=k}} \frac{p_{k, L(f_k), R}(f_k, t_2)}{k!}\overline{p}_{k, S_1, R}(t_1, t_2) & \text{if $l(t_1) = k$.}
  \end{cases}
\]
\end{prop}

\begin{proof}
Given a partial gluing of $t_1$ and $t_2$, let us separately consider the part of the gluing with $t_{11}$ partially glued to $t_2$ and the part of the gluing with $\overline{t}$ partially glued to $t_2$.  The partial gluing of $t_1$ and $t_2$ determines the partial gluings of $t_{11}$ and $t_2$ and of $\overline{t}$ and $t_2$ and the two partial gluings of $t_{11}$ and $t_2$ and $\overline{t}$ and $t_2$ determine a partial gluing of $t_1$ and $t_2$ as long as the target leaves of $t_2$ are disjoint.

It only remains to consider subdivergences.  Internal subdivergences or right-sided subdivergences in partial gluings of $t_{11}$ and $t_2$ or $\overline{t}$ and $t_2$ remain internal or right-sided subdivergences in the partial gluing of $t_1$ and $t_2$.  Likewise internal or right-sided subdivergences in a partial gluing of $t_1$ and $t_2$ must involve a cut edge on the $t_1$ side and that edge is either in $B_+(t_{11})$ or $\overline{t}$ and hence give an internal or right-sided subdivergence in one of the smaller partial gluings.  By the previous proposition, we may use $p$ instead of $\overline{p}$ in the first factor on the right hand side of the first equation.  This proves the first equation.

For the second equation, if we have not glued all leaves of $t_1$ then there is no distinction between $p$ and $\overline{p}$.  If we have glued all leaves of $t_1$, then note that for $R\subseteq L(t_2)$ with $|R|=k$, $p_{k, L(f_k), R}(f_k, t_2)$ is either $0$ or $k!$ depending on whether or not the $k$ leaves of $R$ are exactly the leaves of a subtree of $t_2$ or not.  Using this as an indicator function for when there is a left-sided subdivergence we get the second case of the second equation.  This proves the second equation.
\end{proof}

\begin{thm}
Propositions \ref{prop 1 child} and \ref{prop 2 children} along with the bulleted observations preceding them give a recursive definition and recursive algorithm to compute $p$ and $\overline{p}$ and hence to compute $n$.
\end{thm}

\begin{proof}
Given a $t_1$ with more than one vertex, using Proposition \ref{prop 2 children}, we can reduce to the case that $t_1$ has either one child of the root and the same depth or is some $f_k$.  Using Proposition \ref{prop 1 child} we can reduce $t_1$ with one child of the root to a tree of depth one less.  Returning to Proposition \ref{prop 2 children} we can reduce the number of children of the root without increasing the depth.  Continuing this process we can reduce $t_1$ to a tree of depth 1, that is to some $f_k$.  

Swapping $1$ and $2$, we can do the same for $t_2$, so it only remains to consider the case when both trees are fans which is given in the bulleted observations preceding the propositions.  These bullets also cover the trivial cases.

Finally, as noted above when $t_1$ and $t_2$ both have $k$ leaves then $n(t_1, t_2) = p_{k, L(t_1), L(t_2)}(t_1, t_2)$.
\end{proof}

\subsection{Cut Preprocessing Algorithm}\label{ssec preproc}

The following discussion presents another algorithm for counting subdivergence-free gluings between arbitrary trees. A summary of the discussion can be found in Algorithms ~\ref{alg subfree} and ~\ref{alg subs}.  An implementation of the algorithm is included with the source in the arXiv version of this paper \cite{this}, and can also be found on GitHub at \href{https://github.com/jordanmzlong/Subdivergence-Free-Trees}{github.com/jordanmzlong/Subdivergence-Free-Trees}.

Let $t_1$ and $t_2$ be rooted trees whose leaves are each assigned a label in $\mathbb{Z}_{\geq 1}$, as in Figure ~\ref{fig:coloured_rooted_trees}. We will say that these rooted trees are \emph{coloured}, and refer to the labels as \emph{colours}. We will also let $c(t_1)$ denote the multiset of colours of $t_1$. We will say that a gluing is \emph{colour-preserving} if each of the leaves of $t_1$ are identified with a leaf of $t_2$ with the same colour. This algorithm solves the more general problem of counting subdivergence-free colour-preserving gluings between coloured rooted trees $t_1$ and $t_2$, which we will still denote $n(t_1, t_2)$. Our original problem of counting subdivergence-free gluings between trees can then be solved as a special case where every leaf in $t_1$ and $t_2$ has the same colour.

We will assume in this section that our trees do not have non-root 2-valent vertices that are not incident to leaves. This is still completely general, since if we have two edges $e_1$ and $e_2$ with such a 2-valent vertex between them, a gluing has a subdivergence containing $e_1$ or $e_2$ if and only if the same gluing has a subdivergence containing the $e_1$ after $e_2$ is contracted. Thus, prior to running the algorithm, we can preprocess the trees by contracting such edges.

\begin{figure}
  \centering
  \includegraphics[]{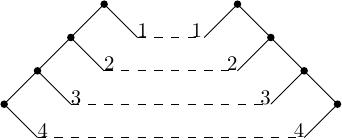}
  \caption{Two coloured rooted trees, with colours 1,2,3,4. The only colour-preserving gluing is drawn with dashed lines. Note that there are no colour-preserving subdivergence-free gluings.}
  \label{fig:coloured_rooted_trees}
\end{figure}

Note that if $c(t_1) \neq c(t_2)$ then there are no colour-preserving gluings, and thus no subdivergence-free colour-preserving gluings. If $c(t_1) = c(t_2)$, the approach of the algorithm is to count the total number of colour-preserving gluings between the trees, and subtract the number of colour-preserving gluings that result in a subdivergence. To count colour-preserving gluings, we multiply together the factorials of the multiplicities in $c(t_1)$. For example, if $c(t_1) = c(t_2) = \{1,1,1,2,2,3\}$, then there are $3!2!1! = 12$ colour-preserving gluings, due to the $3!$ ways to map the leaves coloured $1$, $2!$ ways to map the leaves coloured $2$, and single way to map the leaf coloured $3$. We denote the number of colour-preserving gluings between $t_1$ and $t_2$ by $r(t_1,t_2)$.

Counting colour-preserving gluings that result in a subdivergence is more involved. In order to do so, we first need to define some terms. An edge is \emph{internal} if it is not incident to a leaf.  Given two edges $e_1$ and $e_2$ in a rooted tree $t_1$, if the path from $e_1$ to the root of $t_1$ contains $e_2$, we say that $e_1$ is a \emph{descendant} of $e_2$ and $e_2$ is an \emph{ancestor} of $e_1$. We will call a subset of edges $E$ of a rooted tree $t_1$ \emph{siblings} if no edge in $E$ is a descendant of another edge in $E$. 
The edges which are cut in order to give a subdivergence will be called the edges in the \emph{boundary} of the subdivergence.
Our approach to counting colour-preserving gluings with subdivergences is to iterate through all combinations of nonempty sets of internal siblings of $t_1$, and nonempty sets of internal siblings of $t_2$, and for each combination $S = S_1 \cup S_2$ with $S_1 \subseteq E(t_1)$ and $S_2 \subseteq E(t_2)$, count the number of gluings with the following two properties. 
\begin{itemize}
    \item \emph{Property 1:} each member of $S$ is in the boundary of a subdivergence in the gluing. 
    \item \emph{Property 2:} if any other internal edge $e$ that is in the boundary of a subdivergence in the gluing is added to $S$, $e$ is a descendant of an edge in $S_1$ or $S_2$.
\end{itemize}       

We claim that this process will indeed count all of the gluings with subdivergences.

To see that every gluing with subdivergences is counted, suppose we have a gluing of $t_1$ and $t_2$ that contains a subdivergence. Let $E = E_1 \cup E_2$ be the set of edges that in the boundary of subdivergences in this gluing, with $E_1 \subseteq E(t_1)$ and $E_2 \subseteq E(t_2)$. $E$ can be identified with a combination of sets of siblings $S = S_1 \cup S_2$ by removing descendants in $E_1$ and $E_2$ to give $S_1$ and $S_2$ respectively. Then this gluing would be counted when the algorithm considers $S_1 \cup S_2$. Property 1 is satisfied as every member of this set is in the boundary of a subdivergence in the gluing, and property 2 is satisfied since every other edge that is in the boundary of a subdivergence in the gluing is a descendant of an edge in $S_1$ or $S_2$. 

To see that no gluing would be counted twice, suppose a gluing is counted when considering the combination of siblings $S = S_1 \cup S_2$ and $R = R_1 \cup R_2$, with $R \neq S$. Without loss of generality there is an edge $e$ in $R_1$ not present in $S_1$. By property 1 with $R$, $e$ is in the boundary of a subdivergence in the gluing. However, if $e$ is added to $S$, by property 2 with $S$ we have that $e$ is a descendant of another edge $f$ in $S_1$. Then $f \notin R_1$, which is a set of siblings containing $e$. Also, by property 1 with $S$ we have that $f$ is in the boundary of a subdivergence in the gluing. But $f$ is not a descendant of anything in $R_1$, contradicting property 2 with $R$. Thus $R = S$, so each gluing is only counted when considering a single combination of siblings.

Now, for a given combination of siblings $S_1$ and $S_2$, we explain how to count gluings satisfying properties 1 and 2. Note that for each edge $e_1$ in $S_1$ to be in the boundary of a subdivergence, there must be a distinct edge $e_2$ in $S_2$ such that all of the leaves below $e_1$ can be joined with all of the leaves below $e_2$. Thus, if $|S_1| \neq |S_2|$, no gluing can satisfy property 1. 

If $|S_1| = |S_2|$, we process the trees further, noting that each edge in $S_1$ and $S_2$ is a bridge. Let the components that arise from removing each edge in $S_1$ from $t_1$ be $t_{1_0}, t_{1_1}, \dots, t_{1_{|S_1|}}$, where $t_{1_0}$ is the component containing the root and $t_{1_i}$ is the component containing everything deeper than edge $e_i$ in $t_1$. For each edge $e_i$ in $S_1$, attach a leaf whose colour is unique to $c(t_{1_i})$ to the vertex in $t_{1_0}$ that was adjacent to $e_i$ in $t_1$, and call the resulting coloured rooted tree $u_1$. An example of this process is shown in Figure ~\ref{fig:cut_preprocessing}. In a similar way, form $t_{2_0}, t_{2_1}, \dots, t_{2_{|S_2|}}$ and $u_2$ from $t_2$ and $S_2$. 

\begin{figure}
\begin{subfigure}{.5\textwidth}
  \centering
  \includegraphics[]{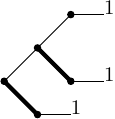}
  \caption{}
  \label{fig:cut_preprocessing_before}
\end{subfigure}%
\begin{subfigure}{.5\textwidth}
  \centering
  \includegraphics[]{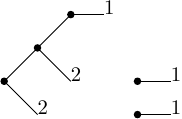}
  \caption{}
  \label{fig:cut_preprocessing_after}
\end{subfigure}%
    \caption{Shown in (A) is a coloured rooted tree $t_1$. The set of siblings to be cut is in bold. (B) shows $u_1$, $t_{1_1}$ and $t_{1_2}$ formed from cutting the bold edges.}
    \label{fig:cut_preprocessing}
\end{figure}

We claim that the number of gluings satisfying properties 1 and 2 with respect to a given $S = S_1 \cup S_2$ is given by $n(u_1,u_2)\cdot \prod_{i=1}^{|S_1|}r(t_{1_i},t_{1_i})$.

To see this, first note that $u_1$ and $u_2$ are representations of $t_1$ and $t_2$ with everything below edges in $S$ contracted. Thus, by counting colour-preserving gluings between $u_1$ and $u_2$ and multiplying this by the number of ways that the contracted parts can be glued together, we count colour-preserving gluings between $t_1$ and $t_2$ that satisfy property 1. Indeed, if we have a colour-preserving gluing between $u_1$ and $u_2$, then by joining the components below each edge in $S_1$ to the corresponding component below an edge in $S_2$ as specified by the gluing, we have a colour-preserving gluing where each edge in $S$ will be in the boundary of a subdivergence. If we restrict this slightly by counting subdivergence-free colour-preserving gluings between $u_1$ and $u_2$ and multiplying this by the number of ways that the contracted parts can be glued together, we count colour-preserving gluings between $t_1$ and $t_2$ that satisfy properties 1 and 2. This makes use of the fact that there are no internal 2-valent vertices in our tree, which guarantees that any edge with an element of $S$ as a descendant will have a different set of leaves below it than any of the components. By avoiding gluings with subdivergences between $u_1$ and $u_2$, this ensures that no edge with a member of $S$ as a descendant will be in the boundary of a subdivergence.

The number of ways that a given component can be glued to another component is the number of colour-preserving gluings between the two components, and thus only depends on the set of colours in each component, rather than their structures. Thus, the number of colour-preserving gluings between two components that can be glued together is the same as the number of colour-preserving gluings between one of the components and another copy of itself.

In practice, we can first consider each tree individually to form a list of all possible preprocessed trees, one for each set of siblings. We can then comb through the lists for $t_1$ and $t_2$ together counting gluings in the manner mentioned above. This approach has the advantage that the list of preprocessed trees can be reused if we wish to count subdivergence-free colour-preserving gluings between $t_1$ or $t_2$ and a different tree.

\medskip

\begin{algorithm}[H]\label{alg subfree}
\SetAlgoLined
\SetKwInOut{Input}{input}\SetKwInOut{Output}{output}
\Input{Two coloured rooted trees $t_1$ and $t_2$}
\Output{Number of subdivergence-free colour-preserving gluings between $t_1$ and $t_2$}
\BlankLine
 \Return ColourPreserving($c(t_1)$, $c(t_2)$) - Subdivergence($t_1$,$t_2$)\;
 \caption{SubdivergenceFree}
\end{algorithm}

\medskip

\begin{algorithm}[H]\label{alg subs}
\SetAlgoLined
\SetKwInOut{Input}{input}\SetKwInOut{Output}{output}
\Input{Two coloured rooted trees $t_1$ and $t_2$}
\Output{Number of colour-preserving gluings with subdivergences between $t_1$ and $t_2$}
\BlankLine
 total := 0\;
 \For{each nonempty set of siblings $S_1$ of $t_1$}{
  \For{each nonempty set of siblings $S_2$ of $t_2$}{
     construct $u_1, t_{1_1}, \dots, t_{1_{|S_1|}}$ from $t_1$ and $S_1$\;
     construct $u_2, t_{2_1}, \dots, t_{2_{|S_2|}}$ from $t_2$ and $S_2$\;
     total += SubdivergenceFree($u_1$,$u_2$)*$\prod_{i=1}^{|S_1|} $ColourPreserving$(c(t_{1_i}),c(t_{1_i}))$\;
  }
 }
 \Return total
 \caption{Subdivergence}
\end{algorithm}

\section{Discussion}\label{sec discussion}

As noted in the introduction gluings of trees are also studied under the name of \emph{tanglegrams}  \cite{10.1093/bioinformatics/btr210, 5551100, BILLEY2017239, ralaivaosaona_et_al:LIPIcs.AofA.2018.32, GESSEL2021105498}, which arose from phylogenetics.  Subdivergence-free is known as irreducible in this context.  Because of the phylogenetic motivation, planar tanglegrams are of particular interest.

In the planar case gluings of trees are also studied under the name of \emph{matings of trees}, often with an interest in limiting properties and applications to Liouville quantum gravity.  See \cite{gwynne2019mating} for a survey.  The mating construction is also considered for its pure combinatorial value, see for instance \cite{biane2021mating}.  Basic to this literature is how mated pairs of Dyck or Motzkin paths correspond to mated trees and how the mated pairs of paths themselves correspond to a quarter plane walk.     Subdivergences do not appear in this context.  Subdivergences in the tree gluing would correspond to places where in the mating of the two Dyck or Motzkin paths two subpaths which both preserve the condition of remaining above and returning to the original level are glued.  In the quarter plane walk this would correspond to a subpath which would remain in the quarter plane if translated to the origin and also end at the origin.  However, between restricting to planar gluings and having other questions in mind, the entire study has a quite different flavour.

\medskip

Our original motivation for studying the enumeration of subdivergence-free gluings of trees was in order to better understand how many Feynman diagrams with cuts can be formed.  From this perspective, the next two questions to consider are cuts which may leave some cycles intact, so rather than gluing two trees we would be gluing two graphs with leaves, and to consider graphs with more subdivergences, that is multiple roots in the two halves.  Additionally, because of anomalous thresholds, see the introduction of \cite{KYcut} and references therein, it would also be interesting to consider cuts into more than two pieces, equivalently, to consider gluing more than two trees or graphs together.

From an enumerative perspective, subdivergence-free gluings of trees provide a generalization of connected permutations that to our knowledge have not been explored until now.  Further classes of trees with nice counting sequences could be investigated.

Another interesting observation is that the second algorithm is based on sets of siblings.
Sets of siblings are precisely admissible cuts in the sense of the Connes-Kreimer coproduct on rooted trees (see \cite{ck0}) and so the second algorithm proceeds essentially by considering the coproducts of $t_1$ and $t_2$ term by term.  The Connes-Kreimer Hopf algebra is not appearing as a renormalization Hopf algebra in this context as we are not working with insertion trees showing the insertion structure of subdivergences in a Feynman diagram, but rather with trees coming directly from a cut Feynman diagram.  None the less it is appearing in a way that has an interesting interplay with the cuts while being different from the Hopf algebraic cointeraction studied in \cite{KYcut}. 

The first algorithm, by contrast, is closely related to the operation $B_+$ which adds a new root to a collection of trees.  $B_+$ is also special in the Connes-Kreimer Hopf algebra context as it is a Hochschild 1-cocycle and is the foundation of the Hopf algebraic approach to Dyson-Schwinger equations and Ward identities, see \cite{Ybook} for an overview.

\end{document}